\documentclass{article}
\usepackage{amsmath,amssymb,amsfonts}
\usepackage{graphicx}
\usepackage{color}

\usepackage{pb-diagram}
\newtheorem{theorem}{Theorem}

\newtheorem{proposition}{Proposition}

\newtheorem{lemma}{Lemma}
\newtheorem{remark}{Remark}

\newenvironment{proof}[1][Proof]{\noindent\textit{#1.} }{\hfill$\Box$\medskip}

\title{New examples of systems of the Kowalevski type}
\author{Vladimir Dragovi\'c (*) and Katarina Kuki\'c (**)}

\date{}

\begin{document}

\maketitle

\medskip

\centerline{(*) Mathematical Institute SANU}

\centerline{Kneza Mihaila 36, 11000 Belgrade, Serbia}

\smallskip

\centerline{Mathematical Physics Group, University of Lisbon}

\smallskip

\centerline{e-mail: {\tt vladad@mi.sanu.ac.rs}}

\smallskip

\centerline{(**) Faculty for Traffic and Transport Engineering}
\centerline{Vojvode Stepe 305, 11000 Belgrade, Serbia}

\smallskip
\centerline{e-mail: {\tt k.mijailovic@sf.bg.ac.rs}}

\begin{abstract}
 A new examples of integrable dynamical systems are
constructed. An integration procedure leading to genus two
theta-functions is presented. It is based on a recent
notion of discriminantly separable polynomials. They have appeared
in a recent reconsideration of the celebrated Kowalevski top, and
their role here is analogue to the situation with the classical
Kowalevski integration procedure.

\end{abstract}

\newpage

\tableofcontents

\newpage

\section{Introduction}\label{sec:intro}

\

A notion of discriminantly separable polynomials has been introduced
recently by one of the authors in \cite{Drag3}. It has been related
there to a new view on the classical integration procedure of
Kowalevski of her celebrated  top (see the original work \cite{Kow},
\cite{Kow1}, a classical presentation in \cite{Gol}, and for modern
approach, see \cite {BRST}, \cite{Aud}, \cite{Dub}). Following the
way the Kowalevski integration procedure has been coded in
\cite{Drag3}, we construct here a new class of integrable dynamical
systems of {\it Kowalevski type}. Their complete
integration procedure parallels the classical one, and leads to the
 formulae in the genus two theta-functions.\\

Let us recall the defintion of
discriminantly separable polynomials from \cite{Drag3}, here for polynomial $\mathcal F(x_1,x_2,s)$ of the second degree in each
of three variables. Polynomial $\mathcal{F}(x_1,x_2,s)$ is discriminantly separable if there exist polynomials $P_1$, $P_2$, $J$
of one variable each of degree not greater than four, such that
$$
\aligned \mathcal D_{x_1}\mathcal F(x_2,s)&=P_2(x_2)J(s)\\
\mathcal D_{x_2}\mathcal F(x_1,s)&=P_1(x_1)J(s)\\
\mathcal D_{s}\mathcal F(x_1, x_2)&=P_1(x_1)P_2(x_2),
\endaligned
$$
where $\mathcal D_s\mathcal F$ denotes the discriminant of $\mathcal
F$ understood as a polynomial in $s$. The discriminant is a
polynomial in the other two variables. When polynomials $P_1,P_2$ and $J$ coincide we say that $\mathcal F$ is strongly discriminantly separable.
Here we distinguish lemma stated in \cite{Drag3} we will refer on later in the text.
\begin{lemma}\label{lemma:drag3} For an arbitrary discriminately separable
polynomial $\mathcal F(x_3,x_1,x_2)$ of the second degree in each of the
variables $x_3, x_1, x_2$, its differential is separable on the
surface $\mathcal F(x_3,x_1,x_2)=0$:
$$
\frac{d \mathcal F}{\sqrt{f_3(x_3)f_1(x_1)f_2(x_2)}}= \frac {dx_3}{\sqrt
{f_3(x_3)}}+\frac {dx_1}{\sqrt {f_1(x_1)}}+\frac {dx_2}{\sqrt
{f_2(x_2)}}.
$$
\end{lemma}

\medskip

Recall now {\it the
Kowalevski fundamental equation}:
\begin{equation}\label{eq:fundKowrel}
Q(s, x_1,x_2)=(x_1-x_2)^2s^2-2R(x_1,x_2)s - R_1(x_1,x_2)=0
\end{equation}
where
$$
\aligned
R(x_1,x_2)&=-x_1^2x_2^2+6l_1x_1x_2+2lc(x_1+x_2)+c^2-k^2\\
R_1(x_1,x_2)&=-6l_1x_1^2x_2^2-(c^2-k^2)(x_1+x_2)^2-4lcx_1x_2(x_1+x_2)+6l_1(c^2-k^2)-4c^2l^2.
\endaligned
$$
$Q(s,x_1,x_2)$ introduced with (\ref{eq:fundKowrel}) as a polynomial in three variables degree two in each of them satisfies
$$\mathcal D_{s}(Q)(x_1,x_2)=4P(x_1)P(x_2)$$
$$\mathcal D_{x_1}(Q)(s,x_2)=-8J(s)P(x_2),\,\mathcal D_{x_2}(Q)(s,x_1)=-8J(s)P(x_1)$$
with
$$P(x_i)=-x_i^4+6l_1 x_i^2 +4lc x_i+c^2-k^2,\, i=1,2$$
$$J(s)=s^3+3l_1s^2+s(c^2-k^2)+3l_1(c^2-k^2)-2l^2c^2.$$
Finally, notice that equations of Kowalevski's top in variables $x_i=p \pm \imath q,\, e_i=x_i^2+c(\gamma_1 \pm \imath\gamma_2),\,i=1,2$
which she introduces, may be rewritten in the form
\begin{equation}\label{eq:Kowsyst2}
\aligned
2 \dot x_1 &= -i (rx_1+c\gamma_3)\\
2 \dot x_2 &= i (rx_2+c\gamma_3)\\
\dot e_1 &= -ir e_1\\
\dot e_2 &=ir e_2
\endaligned
\end{equation}
with two additional differential equations for $\dot r$ and $\dot
\gamma_3$. If we denote
$$ f_1=r x_1 +c \gamma_3,\quad f_2=r x_2 +c \gamma_3,$$
one can easily check that functions $f_1,f_2$ have following property

\begin{equation}\label{eq:Kowsyst3}
\aligned
f_1^2 &= P(x_1) + e_1 (x_1-x_2)^2\\
f_2^2 &= P(x_2) + e_2 (x_1- x_2)^2.
\endaligned
\end{equation}
Generalization of (\ref{eq:Kowsyst2}) and (\ref{eq:Kowsyst3}) represent a base for systems of Kowalevski type we are going to introduce in next section.

\medskip

\section{Subclass of systems of the Kowalevski type}

\

We will start with a modal example.

Let us consider the next system of ordinary differential equations in variables $x_1,x_2,e_1,e_2,r,\gamma_3$
with constant parameter $g_2$ and condition $p \neq 0$:

\begin{equation}\label{eq:ex1-system}
\aligned
\dot p &= p q r\\
\dot q &= -\frac{1}{2} \left((p^2-q^2)r + \gamma_3 \right)\\
\dot r &= -\frac{1}{2} \frac{ p q - q (p^2+q^2)+q \gamma_1 -p \gamma_2 }{p^2}\\
\dot \gamma_1 &= (p^2 + q^2) q r + 2 p r \gamma_2 -q \gamma_3\\
\dot \gamma_2 &= -2 p q^2 r +p \gamma_3 - p r \gamma_1\\
\dot \gamma_3 &= \frac{g_2 p q + 4 q (p^2+q^2)^2 + 4q \gamma_1 (3p^2-q^2) + 4 p \gamma_2 (p^2-3 q^2)}{8p^2}.
\endaligned
\end{equation}

Before we start with an analysis of the first integrals of the system (\ref{eq:ex1-system}),
let us consider the existence of an invariant measure.

\begin{lemma}\label{lemma:invmes} The system (\ref{eq:ex1-system}) for $p \neq 0$ possesses an invariant measure with density
\begin{equation}\label{eq:rho}
\rho =\frac{1}{4p^2}.
\end{equation}
\end{lemma}
\begin{proof}
Rewrite system of equations (\ref{eq:ex1-system}) in the form:
$$\frac{dp}{X_1}=\frac{dq}{X_2}=\frac{dr}{X_3}=\frac{d\gamma_1}{X_4}=\frac{d\gamma_2}{X_5}=\frac{d\gamma_3}{X_6}=dt$$
where
$$
\aligned
X_1&=p q r\\
X_2&=-\frac{1}{2} \left((p^2-q^2)r + \gamma_3 \right)\\
X_3&=-\frac{1}{2} \frac{ p q - q (p^2+q^2)+q \gamma_1 -p \gamma_2 }{p^2}\\
X_4&=(p^2 + q^2) q r + 2 p r \gamma_2 -q \gamma_3\\
X_5&=-2 p q^2 r +p \gamma_3 - p r \gamma_1\\
X_6&=\frac{g_2 p q + 4 q (p^2+q^2)^2 + 4q \gamma_1 (3p^2-q^2) + 4 p \gamma_2 (p^2-3 q^2)}{8p^2}.
\endaligned
$$
Divergence of $\mathbf{X}=(X_1,X_2,X_3,X_4,X_5,X_6)$ is nonzero:
$$\frac{\partial X_1}{\partial p}+\frac{\partial X_2}{\partial q}+\frac{\partial X_3}{\partial r}+\frac{\partial X_4}{\partial \gamma_1}+\frac{\partial X_5}{\partial \gamma_2}+\frac{\partial X_6}{\partial \gamma_3}=2qr.$$
Simple check shows that density function $\rho$ such that
$$\frac{\partial \rho X_1}{\partial p}+\frac{\partial \rho X_2}{\partial q}+\frac{\partial \rho X_3}{\partial r}+\frac{\partial \rho X_4}{\partial \gamma_1}+\frac{\partial \rho X_5}{\partial \gamma_2}+\frac{\partial \rho X_6}{\partial \gamma_3}=0$$
is given by (\ref{eq:rho}).
\end{proof}

\medskip

Now, we are going to focus on the structure of the first integrals
of the system (\ref{eq:ex1-system}). In order to put this question in a wider context, let
us first make a change of variables:
\begin{equation}\label{eq:ex1-change1}
\aligned
 x_1 &= p + \imath q \\
 x_2 &= p - \imath q \\
 e_1 &= x_1^2 + \gamma_1 + \imath \gamma_2\\
 e_2 &= x_2^2 + \gamma_1 - \imath \gamma_2.
\endaligned
\end{equation}

The system (\ref{eq:ex1-system}) after change (\ref{eq:ex1-change1}) becomes

\begin{equation}\label{eq:ex1-system1}
\aligned
\dot x_1 &= -\frac{\imath}{2} (x_1^2 r +\gamma_3)\\
\dot x_2 &= \frac{\imath}{2} (x_2^2 r +\gamma_3)\\
\dot e_1 &= -\imath (x_1+x_2) r e_1\\
\dot e_2 &= \imath (x_1+x_2) r e_2\\
\dot r &= \frac{\imath}{2} \frac{x_1^2 - x_2^2 + 2 e_2 x_1 - 2 e_1 x_2}{(x_1+x_2)^2}\\
\dot \gamma_3 &= -\frac{\imath}{8} \frac{g_2(x_1^2-x_2^2)+8e_2 x_1^3 -8e_1 x_2^3}{(x_1+x_2)^2}.
\endaligned
\end{equation}

\medskip

Now, let's make assumptions for a subclass of systems of ordinary differential equations that will also include our modal example. After introducing such
systems and establish relations that will hold for them in Theorem \ref{th:integrals} we will return to (\ref{eq:ex1-system1}) and show how one can apply
Kowalevski's procedure from \cite {Kow} on modal example and in the same way on a whole class of systems we are going to introduce.\\

\medskip

Suppose, that a  given system in variables $x_1,x_2,e_1,e_2,r,\gamma_3$, after some transformations  reduces
to

\begin{equation}\label{eq:analysis}
\aligned
2 \dot x_1 &= - i f_1\\
2 \dot x_2 &= i f_2\\
\dot e_1 &= -m e_1\\
\dot e_2 &= m e_2
\endaligned
\end{equation}

where
\begin{equation}\label{eq:analysis1}
\aligned
f_1^2 &= P(x_1) + e_1 A(x_1,x_2)\\
f_2^2 &= P(x_2) + e_2 A(x_1, x_2).
\endaligned
\end{equation}

Here $f_i,\,i=1,2$ and $m$ represent functions of system's variables. Notice here that all systems of this type
will have the first integral $$e_1e_2=k^2.$$

Suppose additionally, that the first integrals of the initial system
reduce to a relation
\begin{equation}\label{eq:integral}
 P(x_2)e_1+P(x_1)e_2=C(x_1,x_2)-e_1e_2A(x_1, x_2)
\end{equation}
with $A(x_1,x_2),\,C(x_1,x_2)$ polynomials in two variables and $P(x_i)$ in one.\\

Systems satisfying above assumptions (\ref{eq:analysis}), (\ref{eq:analysis1}) and (\ref{eq:integral}) we will call systems of Kowalevski type.\\

We are looking for possible first integrals of such systems in the form

\begin{equation}\label{eq:firstint}
\aligned
 r^2 &= E + p_2 e_1 + p_1 e_2\\
 r \gamma_3 &= F-q_2e_1 - q_1e_2\\
  \gamma_3^2& = G +r_2e_1 + r_1e_2\\
e_1 \cdot e_2 &= k^2
\endaligned
\end{equation}
where $p_i,q_i,r_i,E,F,G$ are functions of $x_1,x_2$. Subclass of Kowalevski-type systems for which functions $f_i, i=1,2$ are of type
\begin{equation}\label{eq:f_i}
f_i=x_i^{m_i}\cdot r +x_i^{n_i} \cdot \gamma_3, \quad m_i,\,n_i \in \mathbb{Z}  ,\, i=1,2
\end{equation}
is a subject of next theorem. In the next section we will show how one can reconstruct system of equations for that subclass.

\begin{theorem}\label{th:integrals} For a system which reduces
to (\ref{eq:analysis}), (\ref{eq:analysis1}), (\ref{eq:integral}) with functions $f_i$ introduced in (\ref{eq:f_i}), and at least one of conditions $m_1 \neq n_1$ or $m_2 \neq n_2$ is valid,
relations in the form (\ref{eq:firstint}) are satisfied for expressions:

\begin{equation}\nonumber
\aligned p_1 &=\frac{A
x_1^{2n_1}}{(x_1^{m_1}x_2^{n_2}-x_2^{m_2}x_1^{n_1})^2} \quad p_2 =\frac{A
x_2^{2n_2}}{(x_1^{m_1}x_2^{n_2}-x_2^{m_2}x_1^{n_1})^2}\\
q_1 &=\frac{A
x_1^{n_1+m_1}}{(x_1^{m_1}x_2^{n_2}-x_2^{m_2}x_1^{n_1})^2} \quad q_2 =\frac{A
x_2^{n_2+m_2}}{(x_1^{m_1}x_2^{n_2}-x_2^{m_2}x_1^{n_1})^2}\\
r_1 &=\frac{A
x_1^{2m_1}}{(x_1^{m_1}x_2^{n_2}-x_2^{m_2}x_1^{n_1})^2} \quad r_2 =\frac{A
x_2^{2m_2}}{(x_1^{m_1}x_2^{n_2}-x_2^{m_2}x_1^{n_1})^2}\\
E_{i} &=\frac{x_2^{2n_2}P(x_1)+x_1^{2n_1}P(x_2) \pm
B(x_1,x_2) x_1^{n_1}x_2^{n_2}}{(x_1^{m_1}x_2^{n_2}-x_2^{m_2}x_1^{n_1})^2} ,\, i=1,2\\
F_{i} &=
\frac{E_i(x_1^{2m_1}x_2^{2n_2}-x_1^{2n_1}x_2^{2m_2})+x_1^{2n_1}P(x_2)-x_2^{2n_2}P(x_1)}{2x_1^{n_1}x_2^{n_2}
(x_1^{n_1} x_2^{m_2} -x_1^{m_1} x_2^{n_2})} ,\, i=1,2\\
G_i &= \frac{E_i x_1^{m_1} x_2^{m_2}(x_1^{m_1} x_2^{n_2} -
x_1^{n_1} x_2^{m_2})+
x_1^{m_1+n_1}P(x_2)-x_2^{m_2+n_2}P(x_1)}{x_1^{n_1} x_2^{n_2}
(x_1^{m_1} x_2^{n_2}-x_1^{n_1} x_2^{m_2} )} ,\, i=1,2.
\endaligned
\end{equation}
Here by $B(x_1,x_2)$ we denoted
$$B^2(x_1,x_2)=4A(x_1,x_2)C(x_1,x_2)+4P(x_1)P(x_2).$$
\end{theorem}

\begin{proof} Replacing (\ref{eq:firstint}) into condition (\ref{eq:analysis1}) with $f_i=x_i^{m_i}\cdot r +x_i^{n_i} \cdot \gamma_3 $, we get
$r^2x_i^{2m_i}+2r \gamma_3 x_i^{m_i+n_i}+\gamma_3^2x_i^{2n_i}=P(x_i)+e_iA(x_1,x_2)$. Collecting coefficients with $e_i$ we obtain system
\begin{equation}
\aligned
 p_2x_1^{2m_1}-2q_2x_1^{m_1+n_1}+r_2x_1^{2n_1} & =A(x_1,x_2)\\
 p_1x_1^{2m_1}-2q_1x_1^{m_1+n_1}+r_1x_1^{2n_1} & =0\\
 Ex_1^{2m_1}+2Fx_1^{m_1+n_1}+Gx_1^{2n_1} & =P(x_1)\\
 p_1x_2^{2m_2}-2q_1x_2^{m_2+n_2}+r_1x_2^{2n_2} & =A(x_1,x_2)\\
 p_2x_2^{2m_2}-2q_2x_2^{m_2+n_2}+r_2x_2^{2n_2} & =0\\
 Ex_2^{2m_2}+2Fx_2^{m_2+n_2}+G x_2^{2n_2}  & =P(x_2)
\endaligned
\end{equation}

with solutions:

\begin{equation}\nonumber
\aligned
p_1 &=\frac{A(x_1,x_2)x_1^{2n_1}}{(x_1^{m_1}x_2^{n_2}-x_2^{m_2}x_1^{n_1})^2}\\
p_2 &=\frac{A(x_1,x_2)x_2^{2n_2}}{(x_1^{m_1}x_2^{n_2}-x_2^{m_2}x_1^{n_1})^2}\\
r_1 &=\frac{A(x_1,x_2)x_1^{2m_1}}{(x_1^{m_1}x_2^{n_2}-x_2^{m_2}x_1^{n_1})^2}\\
r_2 &=\frac{A(x_1,x_2)x_2^{2m_2}}{(x_1^{m_1}x_2^{n_2}-x_2^{m_2}x_1^{n_1})^2}\\
F &=\frac{(x_2^{2n_2}x_1^{2m_1}-x_1^{2n_1}x_2^{2m_2})E+x_1^{2n_1}P(x_2)-x_2^{2n_2}P(x_1)}
{2(x_1^{2n_1}x_2^{m_2+n_2}-x_2^{2n_2}x_1^{m_1+n_1})}\\
G &=-\frac{(x_2^{m_2+n_2}x_1^{2m_1}-x_2^{2m_2}x_1^{m_1+n_1})E-x_2^{m_2+n_2}P(x_1)+x_1^{m_1+n_1}P(x_2)}{x_1^{2n_1}x_2^{m_2+n_2}-x_2^{2n_2}x_1^{m_1+n_1}}
\endaligned
\end{equation}

The second assumption is that the relation

\begin{equation}\label{eq:r^2g_3^2-(rg_3)^2} (E + p_2 e_1 + p_1 e_2)(G
+r_2e_1 + r_1e_2)-(F-q_2e_1 - q_1e_2)^2=0
\end{equation}

is in the form (\ref{eq:integral}). According to
(\ref{eq:integral}), the  coefficients of $e_i^2$ should vanish, so
we get:
\begin{equation}\nonumber
\aligned
q_1 &=\frac{A(x_1,x_2)x_1^{n_1+m_1}}{(x_1^{m_1}x_2^{n_2}-x_2^{m_2}x_1^{n_1})^2}\\
q_2 &=\frac{A(x_1,x_2)x_2^{n_2+m_2}}{(x_1^{m_1}x_2^{n_2}-x_2^{m_2}x_1^{n_1})^2}.
\endaligned
\end{equation}

Replacing these results into (\ref{eq:r^2g_3^2-(rg_3)^2}) it
becomes\\

$$
\frac{A}{(x_1^{m_1}x_2^{n_2}-x_1^{n_1}x_2^{m_2})^2}\left(Ae_1e_2+P(x_2)e_1+P(x_1)e_2\right)+\varphi(E)=0$$
with $\varphi(E)$, a quadratic function of $E$
$$
\varphi(E)= -E^2\frac{(x_1^{m_1-n_1}-x_2^{m_2-n_2})^2}{4}+E\frac{\frac{P(x_1)}{x_1^{2n_1}}+\frac{P(x_2)}{x_2^{2n_2}}}{2}-\frac{(P(x_1)x_2^{2n_2}-P(x_2)x_1^{2n_1})^2}{4x_1^{2n_1}x_2^{2n_2}(x_1^{n_1}x_2^{m_2}-x_1^{m_1}x_2^{n_2})^2}.
$$

Finally, solving the quadratic equation
$$\varphi(E)=-C\cdot\frac{A}{(x_1^{m_1}x_2^{n_2}-x_1^{n_1}x_2^{m_2})^2},$$
we get the solutions

\begin{equation}\label{eq:E}
E_{i} =\frac{x_2^{2n_2}P(x_1)+x_1^{2n_1}P(x_2) \pm
B(x_1,x_2) x_1^{n_1}x_2^{n_2}}{(x_1^{m_1}x_2^{n_2}-x_2^{m_2}x_1^{n_1})^2} ,\, i=1,2.
\end{equation}

\end{proof}

\medskip

\begin{remark}
The discriminant of
\begin{equation}\label{eq:polF}
\mathcal{F}(x_1,x_2,s)=A(x_1,x_2)s^2+B(x_1,x_2)s+C(x_1,x_2)
\end{equation}
as a polynomial in $s$ is factorizable
$$\mathcal{D}_s\mathcal F (x_1,x_2)=B^2(x_1,x_2)-4A(x_1,x_2)C(x_1,x_2)=4P(x_1)P(x_2).$$
If we choose $A,B,C$ to be coefficients of discriminantly separable polynomial of degree two in each variable then
integration of the systems which satisfy assumption of
Theorem 1, will be performed following Kowalevski's procedure in terms of theta-function of genus
two.
\end{remark}

\

\section{Examples}\label{sec:example}

\

Our modal example (\ref{eq:ex1-system}) after change of variables
(\ref{eq:ex1-change1}) fits into introduced subclass of the systems
of the Kowalevski type  as a system that reduces to
(\ref{eq:analysis}), (\ref{eq:analysis1}), (\ref{eq:integral}) with
\begin{equation}\label{eq:ex1-f_i}
f_1=rx_1^2+\gamma_3, \quad f_2=rx_2^2+\gamma_3
\end{equation} and
\begin{equation}\label{eq:ABCP}
\aligned
A(x_1,x_2)&=(x_1-x_2)^2,\\
B(x_1,x_2)&=-2x_1x_2(x_1+x_2)+\frac{g_2}{2}(x_1+x_2)+g_3,\\
C(x_1,x_2)&=x_1^2 x_2^2+\frac{g_2}{2} x_1 x_2+g_3(x_1+x_2)+\frac{g_2^2}{16},\\
P(x)&=2 x^3-\frac{g_2}{2}x-\frac{g_3}{2}.
\endaligned
\end{equation}

Notice here that $$\mathcal{F}(x_1,x_2,s)=A(x_1,x_2)s^2+B(x_1,x_2)s+C(x_1,x_2)$$ is strongly discriminately separable polynomial with the discriminants
$$\mathcal{D}_s\mathcal F (x_1,x_2)=B^2(x_1,x_2)-4A(x_1,x_2)C(x_1,x_2)=4P(x_1)P(x_2)$$
$$\mathcal{D}_{x_1}\mathcal F(x_2,s)=4P(x_2)P(s)$$
$$\mathcal{D}_{x_2}\mathcal F(x_1,s)=4P(x_1)P(s).$$

With previous assumptions, from Theorem \ref{th:integrals} with
$m_1=m_2=2,n_1=0=n_2$ follows that next relations are satisfied:

\begin{equation}\label{eq:ex1-th}
\aligned
 r^2 &= \frac{2}{x_1 + x_2} +  \frac{e_1}{(x_1+x_2)^2} +  \frac{e_2}{(x_1+x_2)^2}\\
 r \gamma_3 &= \frac{4x_1x_2 -g_2}{4(x_1+x_2)}-\frac{x_2^2 e_1}{(x_1+x_2)^2}-\frac{x_1^2 e_2}{(x_1+x_2)^2}\\
 \gamma_3^2 & = -\frac{x_1x_2g_2}{2(x_1+x_2)}+\frac{x_2^4 e_1}{(x_1+x_2)^2}+\frac{x_1^4 e_2}{(x_1+x_2)^2}+\frac{g_3}{2}\\
e_1 \cdot e_2 &= k^2.
\endaligned
\end{equation}

\medskip

Notice here that singular hyperplane $$\alpha: p=0$$ or after a
change of variables (\ref{eq:ex1-change1})
$$\tilde{\alpha}:x_1+x_2=0$$ divides $\mathbb{R}^6$ into two half-spaces, which are simply
connected. Denote those half-spaces with $H^+$ for $p>0$ and $H^-$
for $p<0$.

From the Jacobi theorem, and using Lemma \ref{lemma:invmes}, we come
to the following

\

\begin{proposition}  The system of equations (\ref{eq:ex1-system1}) is completely integrable on
two invariant simply-connected sets $H^+$ and $H^-$, since it has
the first integrals and invariant relations (\ref{eq:ex1-th}), and
an invariant measure with density $\mu =\frac{1}{(x_1+x_2)^2}$.
\end{proposition}

\medskip

In extension we will show how this subclass of Kowalevski type systems fits into Kowalevski's transforming procedure, (see \cite{Kow}).\\

Multiplying the first and the third relation from (\ref{eq:ex1-th}) and deducting square of the second one obtains relation in the form (\ref{eq:integral}):
\begin{equation}\label{eq:ex1-analysis}
\frac{(x_1-x_2)^2 e_1 e_2 +P(x_2)e_1 +P(x_1)e_2 -C(x_1,x_2)}{(x_1+x_2)^2}=0,
\end{equation}

Since $e_1e_2=k^2$ we get
$$(\sqrt{e_1} \sqrt{P(x_2)} \pm \sqrt{e_2} \sqrt{P(x_1)})^2=-k^2 (x_1-x_2)^2+C(x_1,x_2)\pm 2 \sqrt{P(x_1)P(x_2)}k$$
The last relations lead to
\begin{equation}\label{eq:ex1-rel1}
\left(\sqrt{e_1}\frac{\sqrt{P(x_2)}}{x_1-x_2}+
\sqrt{e_2}\frac{\sqrt{P(x_1)}}{x_1-x_2}\right)^2=(s_1-k)(s_2+k)
\end{equation}
and
\begin{equation}\label{eq:ex1-rel2}
\left(\sqrt{e_1}\frac{\sqrt{P(x_2)}}{x_1-x_2}-
\sqrt{e_2}\frac{\sqrt{P(x_1)}}{x_1-x_2}\right)^2=(s_1+k)(s_2-k)
\end{equation}
where $s_1, s_2$ are the solutions of the quadratic equation
\begin{equation}\label{eq:ex1-Kowrel}
\mathcal{F}(x_1,x_2,s)=A(x_1,x_2) s^2 +B(x_1,x_2)s + C(x_1,x_2)=0,
\end{equation}
with $A,B,C$ introduced in (\ref{eq:ABCP}).\\

Notice also that relations (\ref{eq:ex1-rel1}) and
(\ref{eq:ex1-rel2}) lead to a morphism between two two-valued
Buchstaber-Novikov groups, as it has been explained in \cite {Drag3}
(for basic notions and examples of the theory of $n$-valued
Buchstaber-Novikov groups see \cite {Buc}).

From (\ref{eq:ex1-rel1}) and (\ref{eq:ex1-rel2}) we get

\begin{equation}\label{eq:ex1-rel3}
\aligned
2\sqrt{e_1}\frac{\sqrt{P(x_2)}}{x_1-x_2}&=\sqrt{(s_1-k)(s_2+k)}+\sqrt{(s_1+k)(s_2-k)}\\
2\sqrt{e_2}\frac{\sqrt{P(x_1)}}{x_1-x_2}&=\sqrt{(s_1-k)(s_2+k)}-\sqrt{(s_1+k)(s_2-k)}.
\endaligned
\end{equation}

Solutions $s_i$ of the quadratic equation $\mathcal{F}(x_1,x_2,s)$ are
$$
\aligned
s_1&=\frac{-B(x_1,x_2)-\sqrt{B^2(x_1,x_2)-4(x_1-x_2)^2C(x_1,x_2)}}{2(x_1-x_2)^2}\\
s_2&=\frac{-B(x_1,x_2)+\sqrt{B^2(x_1,x_2)-4(x_1-x_2)^2C(x_1,x_2)}}{2(x_1-x_2)^2}.
\endaligned
$$
Using the Vi\`{e}te formulae and the discriminant separability
condition we get
\begin{equation}\label{eq:ex1-Viete s_i}
\aligned
s_1+s_2 &= -\frac{B(x_1,x_2)}{(x_1-x_2)^2}\\
s_2-s_1 &= \frac{\sqrt{4P(x_1)P(x_2)}}{(x_1-x_2)^2}.
\endaligned
\end{equation}
Since
$$\dot x_1 = -\frac{\imath}{2} (x_1^2 r +\gamma_3),$$
then
\begin{equation}\label{eq:ex1-rel4}
-4 \dot x_1^2 = x_1^4 r^2+2x_1^2 r \gamma_3 +\gamma_3^2.
\end{equation}
Substituting $r^2,\,r\gamma_3,\,\gamma_3^2$ from (\ref{eq:ex1-th}) into (\ref{eq:ex1-rel4}) we obtain
\begin{equation}\label{eq:ex1-rel5}
\aligned
-4 \dot x_1^2 &= 2x_1^3-\frac{g_2}{2}x_1-\frac{g_3}{2}+(x_1-x_2)^2e_1\\
&=P(x_1)+(x_1-x_2)^2 e_1.
\endaligned
\end{equation}
The same is
$$-4 \dot x_2^2=P(x_2)+(x_1-x_2)^2 e_2.$$

Using (\ref{eq:ex1-rel3}) and (\ref{eq:ex1-Viete s_i}) finally we get
$$
\aligned
-4 \dot x_1^2 &=P(x_1)+(x_1-x_2)^2 e_1\\
&=\frac{(x_1-x_2)^4}{4P(x_2)}[(s_1-s_2)^2+(\sqrt{(s_1-k)(s_2+k)}+\sqrt{(s_1+k)(s_2-k)})^2]\\
&=\frac{P(x_1)}{(s_1-s_2)^2}[\sqrt{(s_1-k)(s_1+k)}+\sqrt{(s_2-k)(s_2+k)}]^2
\endaligned
$$
and
$$-4 \dot x_2^2=\frac{P(x_2)}{(s_1-s_2)^2}[\sqrt{(s_1-k)(s_1+k)}-\sqrt{(s_2-k)(s_2+k)}]^2.$$

>From the last two equations, it follows
$$
\aligned
\frac{dx_1}{\sqrt{P(x_1)}}+\frac{dx_2}{\sqrt{P(x_2)}}&=\imath\frac{\sqrt{(s_1-k)(s_1+k)}}{s_1-s_2}dt\\
\frac{dx_1}{\sqrt{P(x_1)}}-\frac{dx_2}{\sqrt{P(x_2)}}&=\imath\frac{\sqrt{(s_2-k)(s_2+k)}}{s_1-s_2}dt.
\endaligned
$$

Now we will make use of a Lemma \ref{lemma:drag3} stated in Introduction.\\

The strong discriminant separability of polynomial $ \mathcal{F}(x_1,x_2,s)$ implies
\begin{equation}\label{eq:ex1-kowch}
\aligned
\frac{dx_1}{\sqrt{P(x_1)}}+\frac{dx_2}{\sqrt{P(x_2)}}&=\frac{ds_1}{\sqrt{P(s_1)}}\\
\frac{dx_1}{\sqrt{P(x_1)}}-\frac{dx_2}{\sqrt{P(x_2)}}&=-\frac{ds_2}{\sqrt{P(s_2)}}.
\endaligned
\end{equation}

This way, we can finally conclude
\begin{proposition}
The system of differential equations defined by (\ref{eq:ex1-system1}) is
integrated through the solutions of the system
\begin{equation}\label{eq:ex1-kowvareq}
\aligned
\frac{ds_1}{\sqrt{\Phi(s_1)}}+\frac{ds_2}{\sqrt{\Phi(s_2)}}&=0\\
\frac{s_1\,ds_1}{\sqrt{\Phi(s_1)}}+\frac{s_2\,ds_2}{\sqrt{\Phi(s_2)}}&=\imath\,dt,
\endaligned
\end{equation}
where
$$
\Phi(s)=P(s)(s-k)(s+k).
$$
It is
linearized on the Jacobian of the curve
$\Gamma:\,
y^2=\Phi(s)$.
\end{proposition}

As a result of Theorem \ref{th:integrals}, there is another system, this time with two constant parameters $g_2$ and $g_3$ which also
reduces to (\ref{eq:analysis}), (\ref{eq:analysis1}),
(\ref{eq:integral}) and satisfies relations in the form (\ref{eq:ex1-th}):
\begin{equation}\nonumber
\aligned
r^2 &=\frac{2(x_1+x_2)(x_1^2+x_2^2-\frac{g_2}{2})-2g_3}{(x_1^2-x_2^2)^2}+\frac{e_1}{(x_1+x_2)^2}+\frac{e_2}{(x_1+x_2)^2}\\
r \gamma_3 &=\frac{(x_1+x_2)^3g_2+4(x_1^2+x_2^2)g_3-4x_1x_2(x_1+x_2)^3}{4(x_1^2-x_2^2)^2}-\frac{x_2^2e_1}{(x_1+x_2)^2}-\frac{x_1^2e_2}{(x_1+x_2)^2}\\
\gamma_3^2 &= \frac{-x_1x_2(x_1+x_2)(x_1^2+x_2^2)g_2-(x_1^2+x_2^2)^2g_3+8x_1^3x_2^3(x_1+x_2)}{2(x_1^2-x_2^2)^2}\\
&-\frac{x_2^4e_1}{(x_1+x_2)^2}-\frac{x_1^4e_2}{(x_1+x_2)^2}.
\endaligned
\end{equation}

Differentiating first and third of previous relations, with $\dot x_i$ and $\dot e_i$ given by (\ref{eq:analysis}), (\ref{eq:ex1-f_i}), we get expressions for $\dot r$ and $\dot \gamma_3$. Replacing these values into differentiated second relation we get so far unknown function $m$ and finally we get system of equations:
$$
\aligned
\dot x_1 &=-\frac{\imath}{2}(x_1^2 r+\gamma_3)\\
\dot x_2 &=\frac{\imath}{2}(x_2^2 r+\gamma_3)\\
\dot e_1 &= -m e_1\\
\dot e_2 &= me_2\\
\dot r &= -\frac{\imath e_1(r(x_2-x_1)-\imath m)}{2(x_1+x_2)^2r}-\frac{\imath e_2(r(x_2-x_1)+\imath m)}{2(x_1+x_2)^2r}\\
&+\frac{\imath g_2(3(x_1^2+x_2^2)r-2x_1x_2r+4\gamma_3)}{4(x_2-x_1)^3(x_2+x_1)r} +\frac{2\imath((x_1^2-x_1x_2+x_2^2)r+\gamma_3)g_3}{r(x_1+x_2)^2(x_2-x_1)^3}\\&
-\frac{\imath((x_1^4+x_2^4+6x_1^2x_2^2)r+2(x_1+x_2)^2\gamma_3)}{2(x_2-x_1)^3(x_1+x_2)r}
\endaligned
$$
$$
\aligned
\dot \gamma_3&= \frac{\imath((x_1+x_2)x_2r+m\imath x_2+2\gamma_3)x_2^3e_1}{\gamma_3(x_1+x_2)^2}-\frac{\imath((x_1+x_2)x_1r+m\imath x_1+2\gamma_3)x_1^3e_2}{\gamma_3(x_1+x_2)^2}\\
&-\frac{\imath(rx_1x_2+\gamma_3)(x_1-\imath x_2)(x_1+\imath x_2)x_1x_2g_3}{\gamma_3(x_1+x_2)^2(x_1-x_2)^3}\\
&-\frac{\imath((x_1^4+6x_1^2x_2^2+x_2^4)\gamma_3+2x_1^2x_2^2(x_1+x_2)^2r)g_2}{8\gamma_3(x_1+x_2)(x_1-x_2)^3}\\
&+\frac{\imath(3(x_1^2+x_2^2)\gamma_3-2x_1x_2\gamma_3+4x_1^2x_2^2r)x_1^2x_2^2}{\gamma_3(x_1+x_2)(x_1-x_2)^3}.
\endaligned
$$
where
$$
\aligned
m &= (x_1+x_2)r\imath+\frac{1}{4\imath(x_1+x_2)^2(x_1-x_2)^3\left((x_2^2r+\gamma_3)^2e_1
-(x_1^2r+\gamma_3)^2e_2\right)}\\
&\cdot \big[\big((x_1+x_2)^3(2x_1^2x_2^2(x_1+x_2)^2r^3+\gamma_3(6x_1^2x_2(1+x_2)+4x_1x_2(x_1^2+x_2^2)+5x_2^4\\
&-3x_2^2+2x_1^4)r^2+2\gamma_3^2(5(x_1^2+x_2^2)+2x_1x_2)r+8\gamma_3^3)g_2
+(8x_1^2x_2^2(x_1^2+x_2^2)\\
&\cdot(x_1+x_2)^2r^3+8\gamma_3(2x_2^6-x_1^2x_2^2+x_1^6+2x_1^3x_2^2+3x_1^4x_2^2+2x_2^3x_1^2
-x_1x_2^3-x_2^4\\
&+2x_1^4x_2+3x_2^5x_1+2x_1^5x_2+4x_2^4x_1^2+6x_2^3x_1^3)r^2+8\gamma_3^2
(4x_2^4+4x_1^2x_2^2-x_2^2\\
&+6x_1^3x_2+6x_1x_2^3+3x_1^4+2x_2x_1^2)r+16\gamma_3^3(x_1+x_2)^2)g_3\\
&-4(x_1+x_2)^3(8r^3x_1^4x_2^4+2x_1x_2^2\gamma_3(-2x_1^2x_2+5x_1x_2^2-2x_2+2x_2^3+5x_1^3+4x_1^2)r^2\\
&+\gamma_3^2(2x_1^4+4x_1^3x_2+x_2^4+x_2^2-2x_1^2x_2+4x_1x_2^3+14x_1^2x_2^2)r+2\gamma_3^3(x_1+x_2)^2)\big)\big]
\endaligned
$$

\medskip

We will now give one more example of system belonging to previously introduced subclass for which relations obtained in Theorem \ref{th:integrals} represent actually set of the first integrals. We will replace Kowalevski's fundamental equation (\ref{eq:fundKowrel}) by strongly discriminantly separable polynomial (\ref{eq:polF}) with coefficients
\begin{equation}\label{eq:coeffF}
\aligned
A(x_1,x_2)&=(x_1-x_2)^2,\\
B(x_1,x_2)&=2x_1x_2(x_1+x_2)+2ax_1x_2+b(x_1+x_2)+2c,\\
C(x_1,x_2)&=x_1^2x_2^2-bx_1x_2-2c(x_1+x_2)+\frac{b^2}{4}-ac.
\endaligned
\end{equation}

Then, discriminants od $\mathcal{F}$ are:
$$
\aligned
\mathcal D_{s}\mathcal F(x_1, x_2)&=P(x_1)P(x_2),\\
\mathcal D_{x_1}\mathcal F(x_2,s)&=P(x_2)P(s),\\
\mathcal D_{x_2}\mathcal F(x_1,s)&=P(x_1)P(s),
\endaligned
$$
with polynomial $P$ of third degree
$$P(x)=2x^3+ax^2+bx+c.$$

Applying result of Theorem \ref{th:integrals} for $m_i=1,n_i=0, i=1,2$ like in Kowalevski's case, but on previously introduced polynomials $A,C$ and $P$ we get that system reduces to (\ref{eq:analysis}), (\ref{eq:analysis1}), (\ref{eq:integral}) also satisfies relations (\ref{eq:firstint}) with
$$
\aligned
p_1&=1,\,p_2=1\\
q_1&=x_1,\,q_2=x_2\\
r_1&=x_1^2,\,r_2=x_2^2
\endaligned
$$
and
\begin{equation}\label{eq:ex2E1}
\aligned
E_1&=2x_1+2x_2+a\\
F_1&=-x_1x_2+\frac{b}{2}\\
G_1&=c,
\endaligned
\end{equation}
or
\begin{equation}\label{eq:ex2E2}
\aligned
E_2&=\frac{2(x_1+x_2)(x_1^2+x_2^2)+a(x_1+x_2)^2+2b(x_1+x_2)+4c}{(x_1-x_2)^2}\\
F_2&=-\frac{2x_1x_2(3x_1^2+2x_1x_2+3x_2^2)+4ax_1x_2(x_1+x_2)}{2(x_1-x_2)^2}\\
&+\frac{b(x_1^2+6x_1x_2+x_2^2)+4c(x_1+x_2)}{2(x_1-x_2)^2}\\
G_2&=\frac{4x_1^2x_2^2(x_1+x_2+a)+2bx_1x_2(x_1+x_2)+c(x_1+x_2)^2}{(x_1-x_2)^2}.
\endaligned
\end{equation}
\begin{lemma} For choice of values $E_1,F_1,G_1$ relations (\ref{eq:firstint}) are
\begin{equation}\label{eq:ex2firstint}
\aligned
r^2&=2(x_1+x_2)+e_1+e_2+a\\
r\gamma_3&=-x_1x_2+\frac{b}{2}-x_2e_1-x_1e_2\\
\gamma_3^2&=x_2^2e_1+x_1^2e_2+c\\
e_1e_2&=d^2
\endaligned
\end{equation}
and represent the first integrals of following system:
\begin{equation}\label{eq:ex2system}
\aligned
\dot x_1&=-\frac{i}{2}(rx_1+\gamma_3)\\
\dot x_2&=\frac{i}{2}(rx_2+\gamma_3)\\
\dot e_1&=-ire_1\\
\dot e_2&=ire_2\\
\dot r&=\frac{i}{2}(x_2-x_1+e_2-e_1)\\
\dot \gamma_3&=\frac{i}{2}(e_1x_2-e_2x_1).
\endaligned
\end{equation}
\end{lemma}

\begin{proof}
We start with assumptions that system of equations is of the form (\ref{eq:analysis}) with $f_i=rx_i+\gamma_3,i=1,2$.
Like in previous example, by differentiating the first and third of relations (\ref{eq:ex2firstint}) we get
$$
\aligned
\dot r&=\frac{i}{2}(x_2-x_1)+\frac{m}{2r}(e_2-e_1)\\
\dot \gamma_3&=\frac{e_1x_2}{2\gamma_2}(irx_2+i\gamma_3-mx_2)+\frac{e_2x_1}{2\gamma_3}(mx_1-irx_1-i\gamma_3).
\endaligned
$$
Then replacing previously obtained values for $\dot r$ and $\dot \gamma_3$ into differentiated second relation
from (\ref{eq:ex2firstint}), we get that it will be identically satisfied for function
$$m=ir$$
what brings us to system (\ref{eq:ex2system}).
\end{proof}

\medskip

\begin{lemma} The system (\ref{eq:ex2system}) preserves the standard measure.
\end{lemma}
\begin{proof}
As usually, the system (\ref{eq:ex2system}) can be rewritten in a
more compact form:
$$\frac{dx_1}{X_1}=\frac{dx_2}{X_2}=\frac{dr}{X_3}=\frac{de_1}{X_4}=\frac{de_2}{X_5}=\frac{d\gamma_3}{X_6}=dt$$
with
$$
\aligned
X_1&=-\frac{i}{2}(rx_1+\gamma_3)\\
X_2&=\frac{i}{2}(rx_2+\gamma_3)\\
X_3&=\frac{i}{2}(x_2-x_1+e_2-e_1)\\
X_4&=-ire_1\\
X_5&=ire_2\\
X_6&=\frac{i}{2}(e_1x_2-e_2x_1).
\endaligned
$$

Then the divergence of $\mathbf{X}=(X_1,X_2,X_3,X_4,X_5,X_6)$ is
zero:
$$
\frac{\partial X_1}{\partial x_1}+\frac{\partial X_2}{\partial x_2}+\frac{\partial X_3}{\partial r}+\frac{\partial X_4}{\partial e_1}+\frac{\partial X_5}{\partial e_2}+\frac{\partial X_6}{\partial \gamma_3}=0
$$
so, the standard measure is preserved.
\end{proof}

\medskip

The standard measure is invariant under the flow associated with
(\ref{eq:ex2system}), which makes the system of ordinary
differential equations (\ref{eq:ex2system}) with four first
integrals (\ref{eq:ex2firstint}) completely integrable by Jacobi's
theorem.

\medskip

We have shown that our modal example and the class of systems we
have considered in this paper share many interesting properties,
typical for completely integrable Hamiltonian systems. Through the
connection with discriminantly separable polynomials, they are
particularly close to the celebrated Kowalevski top. Thus, as an
important question, it remains to be seen if they admit a Poisson
structure in which they are Hamiltonian. The question of physical or
mechanical interpretation of such systems is not less interesting.

\subsection*{Acknowledgements}

The authors use the opportunity to thank Prof. B. Gaji\'{c},
Prof. A. Borisov and the referee for stimulating comments and
suggestions. The research was partially supported by the Serbian Ministry of
Science and Technology, Project 174020 {\it Geometry and Topology of
Manifolds, Classical Mechanics and Integrable Dynamical Systems} and
by the Mathematical Physics Group of the University of Lisbon,
Project \emph{ Probabilistic approach to finite and infinite
dimensional dynamical systems, PTDC/MAT/104173/2008}.

\newpage\thispagestyle{empty}
\vspace*{20mm}

\end{document}